\documentclass{baustms}

\usepackage{amsmath,amssymb,amsthm}
\usepackage{mathtools}

\DeclareMathOperator{\ols}{\text{\scshape{o}}}
\DeclareMathOperator{\els}{\text{\scshape{e}}}
\DeclareMathOperator{\rols}{\text{\scshape{ro}}}
\DeclareMathOperator{\rels}{\text{\scshape{re}}}
\DeclareMathOperator{\cols}{\text{\scshape{co}}}
\DeclareMathOperator{\cels}{\text{\scshape{ce}}}
\DeclareMathOperator{\sols}{\text{\scshape{so}}}
\DeclareMathOperator{\sels}{\text{\scshape{se}}}
\newcommand\rowperm[1]{\sigma^{\textrm{row}}_{#1}}
\newcommand\colperm[1]{\sigma^{\textrm{col}}_{#1}}
\newcommand\symperm[1]{\sigma^{\textrm{sym}}_{#1}}
\def\sign{\zeta}
\def\rowpar{\pi_{\textrm{row}}}
\def\colpar{\pi_{\textrm{col}}}
\def\sympar{\pi_{\textrm{sym}}}
\def\rowpar{\pi_{\textrm{row}}}
\def\colpar{\pi_{\textrm{col}}}
\def\sympar{\pi_{\textrm{sym}}}
\usepackage{multicol}

\newtheorem{thrm}{Theorem}

\newtheorem{corol}[thrm]{Corollary}
\newtheorem{lemma}[thrm]{Lemma}

\def\tspacer{{\vrule height 2.75ex width 0ex depth0ex}}
\def\bspacer{{\vrule height 0ex width 0ex depth1.2ex}}

\def\clap#1{\hbox to 0pt{\hss #1\hss}}

\def\dfrac#1#2{\lower0.15ex\hbox{\large$\frac{#1}{#2}$}} 
\renewcommand{\geq}{\geqslant}

\renewcommand{\ge}{\geqslant}
\renewcommand{\le}{\leqslant}

\def\iwb{\rightarrow\infty}
\def\sym{\mathcal{S}}

\def\eref#1{$(\ref{#1})$}

\def\lref#1{Lemma~$\ref{#1}$}
\def\tref#1{Theorem~$\ref{#1}$}
\def\cyref#1{Corollary~$\ref{#1}$}
\def\Tref#1{Table~$\ref{#1}$}

\def\Z{\mathbb{Z}}

\begin{document}

\runningtitle{Parity of Latin squares}
\title{There are asymptotically the same number
of Latin squares of each parity}

\author[1]{Nicholas J. Cavenagh}
\address[1]{Department of Mathematics,
University of Waikato,
Private Bag 3105,
Hamilton, New Zealand.
\email{nickc@waikato.ac.nz}
}

\cauthor
\author[2]{Ian M. Wanless}

\address[2]{School of Mathematical Sciences,
Monash University,
Vic 3800, Australia.
\email{ian.wanless@monash.edu}
}

\authorheadline{N. J. Cavenagh and I. M. Wanless}


\support{Research supported by ARC grant FT110100065.}

\begin{abstract}
A Latin square is reduced if its first row and column are in natural
order.  For Latin squares of a particular order $n$ there are four
possible different parities.  We confirm a conjecture of Stones and
Wanless by showing asymptotic equality between the numbers of reduced
Latin squares of each possible parity as the order $n\iwb$.
\end{abstract}

\classification{05B15}

\keywords{Latin square; parity; Alon-Tarsi conjecture; row cycle}

\maketitle

\section{Introduction}\label{s:intro} 

The parity of permutations plays a fundamental role in group theory.
Latin squares can be thought of as two dimensional permutations and
they also have a notion of parity. A Latin square has three attributes
each of which can be even or odd, although any two of these attributes
determines the third. There are thus four different parities that
Latin squares of a given order may have. These parities account, for
example, for the fragmentation of switching graphs \cite{KMOW14,Wan04}
and the failure of certain topological biembeddings \cite{LDGG09}.
They can also assist in diagnosing symmetries of Latin squares
\cite{Kot12}.

Unlike what happens for permutation groups, there can be different
numbers of Latin squares of each parity. This difference is central to
a famous conjecture by Alon-Tarsi~\cite{AT92} which has ramifications
well beyond its apparent scope \cite{HR94,KL15}.  Nevertheless,
numerical evidence \cite{KMOW14,SW12,Wan04} suggests that within
several natural classes of Latin squares there are very close to the
same number of each parity.  The present note and 
\cite{Alp16} are the first to prove parities are asymptotically 
equinumerous (although \cite{SW12} did show a weaker result in this
direction). An advantage of the present work over \cite{Alp16} is
that we prove a non-trivial result for all orders, whilst \cite{Alp16}
only applies to even orders.

A \emph{Latin square} of order $n$ is an $n \times n$ array
of $n$ symbols such that each symbol occurs exactly once
in each row and exactly once in each column.  We will take the symbol
set to be $[n]:=\{1,2,\dots,n\}$, matching the row and column indices.
A Latin square is \emph{normalised} if the first row is
$(1,2,\dots,n)$.  A Latin square is \emph{reduced} if the first row
is $(1,2,\dots,n)$ and the first column is $(1,2,\dots,n)^T$.  A
Latin square $L=(l_{ij})$ is \emph{unipotent} if
$l_{11}=l_{22}=\dots=l_{nn}$.

Suppose $P$ is a property of Latin squares of order $n$.  Let $L^P_n$,
$R^P_n$ and $U^P_n$ be the numbers respectively of Latin squares,
reduced Latin squares and normalised unipotent Latin squares of order
$n$ with property $P$.  If $P$ is omitted we count the whole class.


Let $\sym_n$ denote the permutations of $[n]$ and $\sign:\sym_n\mapsto\Z_2$
the usual sign homomorphism with kernel the alternating group.
Given a Latin square $L=(l_{ij})$ of order $n$, we can identify the
following $3n$ permutations in $\sym_n$.  For all $i\in[n]$ define
$\rowperm{i}$ by $\rowperm{i}(j)=l_{ij}$.  For all $j \in [n]$ define
$\colperm{j}$ by $\colperm{j}(i)=l_{ij}$.  For all $\ell \in [n]$
define $\symperm{\ell}$ such that $\symperm{\ell}(i)$ is equal to the
$j$ for which $l_{ij}=\ell$.  We call $\rowpar:=\sum_i
\sign(\rowperm{i})$, $\,\colpar:=\sum_j \sign(\colperm{j})$ and
$\sympar:=\sum_{\ell} \sign(\symperm{\ell})$ the \emph{row-parity},
\emph{column-parity} and \emph{symbol-parity} of $L$, respectively.  
A Latin square is called \emph{even} or \emph{odd} if
$\rowpar+\colpar\equiv 0$ or $1$ mod $2$, respectively.  
A Latin square is
called \emph{row-even} or \emph{row-odd} if $\rowpar\equiv0$ or
$1$, respectively.  A Latin square is called
\emph{column-even} or \emph{column-odd} if $\colpar\equiv0$ or
$1$, respectively.  A Latin square is called
\emph{symbol-even} or \emph{symbol-odd} if $\sympar\equiv0$ or
$1$, respectively.  We define the properties:
\begin{multicols}{2}{
\begin{itemize}
  \item[$\els$]= ``is an even Latin square''
  \item[$\ols$]= ``is an odd Latin square''
  \item[$\rels$]= ``is a row-even Latin square''
  \item[$\rols$]= ``is a row-odd Latin square''
  \item[$\cels$]= ``is a column-even Latin square''
  \item[$\cols$]= ``is a column-odd Latin square''
  \item[$\sels$]= ``is a symbol-even Latin square''
  \item[$\sols$]= ``is a symbol-odd Latin square''
\end{itemize}
}
\end{multicols}

We define the \emph{parity} of a Latin square $L$ to be the ordered triple
$\pi=\rowpar \colpar \sympar$. Writing $\pi$ as a superscript denotes that
we are restricting to Latin squares with parity $\pi$.
Some of the basic relationships that are
proved in \cite{SW12} are summarised in \Tref{T:relat}.

\begin{table}
\begin{center}
\begin{tabular}{|l|}
\hline\tspacer
If $n \equiv 0$ or $1 \pmod 4$ \bspacer\\
\hline\tspacer
$R^{\els}_n=R^{\sels}_n=R^{000}_n+R^{110}_n$ \bspacer\\
$R^{\ols}_n=R^{\sols}_n=R^{011}_n+R^{101}_n$ \bspacer\\
\hline\tspacer
$U^{\els}_n=R^{\cels}_n=R^{000}_n+R^{101}_n$ \bspacer\\
$U^{\ols}_n=R^{\cols}_n=R^{011}_n+R^{110}_n$ \bspacer\\
\hline\tspacer
$R^{\rels}_n=R^{000}_n+R^{011}_n=U^{\els}_n$ \bspacer\\
$R^{\rols}_n=R^{101}_n+R^{110}_n=U^{\ols}_n$ \bspacer\\
\hline
\tspacer\tspacer
$R_n^{111}=R_n^{100}=R_n^{010}=R_n^{001}=0$ \bspacer\\
$R_n^{011}=R_n^{101}$ \bspacer\\
$R_n^{011}=R_n^{101}=R_n^{110}$ when $n$ is even \bspacer\\
\hline
\end{tabular}
\quad
\begin{tabular}{|l|}
\hline\tspacer
If $n \equiv 2$ or $3 \pmod 4$ \bspacer\\
\hline\tspacer
$R^{\els}_n=R^{\sols}_n=R^{111}_n+R^{001}_n$ \bspacer\\
$R^{\ols}_n=R^{\sels}_n=R^{100}_n+R^{010}_n$ \bspacer\\
\hline\tspacer
$U^{\els}_n=R^{\cols}_n=R^{111}_n+R^{010}_n$ \bspacer\\
$U^{\ols}_n=R^{\cels}_n=R^{100}_n+R^{001}_n$ \bspacer\\
\hline\tspacer
$R^{\rols}_n=R^{111}_n+R^{100}_n=U^{\els}_n$ \bspacer\\
$R^{\rels}_n=R^{010}_n+R^{001}_n=U^{\ols}_n$ \bspacer\\
\hline
\tspacer\tspacer
$R_n^{000}=R_n^{011}=R_n^{101}=R_n^{110}=0$ \bspacer\\
$R_n^{100}=R_n^{010}$ \bspacer\\
$R_n^{100}=R_n^{010}=R_n^{001}$ when $n$ is even \bspacer\\
\hline
\end{tabular}
\caption{\label{T:relat}Table of identities.}
\end{center}
\end{table}

We use standard `$\sim$', `little-$o$', `big-$O$' and `big-$\Omega$' 
asymptotic notation, always with the order of our Latin squares
$n\iwb$.
The aim of this note is to confirm a conjecture from \cite{SW12}
by proving the following result.

\begin{thrm}\label{t:asymequal}
As $n\iwb$,
\begin{align*}
\begin{rcases}
\;L_n^{000}\sim L_n^{011}\sim L_n^{101}\sim L_n^{110} \sim \dfrac14 L_n\\
R_n^{000}\sim R_n^{011}\sim R_n^{101}\sim R_n^{110} \sim \dfrac14 R_n\\
U_n^{000}\sim U_n^{011}\sim U_n^{101}\sim U_n^{110} \sim \dfrac14 U_n\\
\end{rcases}&\hbox{\rm for }n\equiv0,1\pmod 4,\\
\begin{rcases}
\;L_n^{111}\sim L_n^{100}\sim L_n^{010}\sim L_n^{001} \sim \dfrac14 L_n\\
R_n^{111}\sim R_n^{100}\sim R_n^{010}\sim R_n^{001} \sim \dfrac14 R_n\\
U_n^{111}\sim U_n^{100}\sim U_n^{010}\sim U_n^{001} \sim \dfrac14 U_n\\
\end{rcases}&\hbox{\rm for }n\equiv2,3\pmod 4,\\
L_n^{\els}\sim L_n^{\ols}\sim 
L_n^{\rels}\sim L_n^{\rols}\sim
L_n^{\cels}\sim L_n^{\cols}&\sim 
L_n^{\sels}\sim L_n^{\sols}\sim \dfrac12 L_n,\\
R_n^{\els}\sim R_n^{\ols}\sim 
R_n^{\rels}\sim R_n^{\rols}\sim
R_n^{\cels}\sim R_n^{\cols}&\sim 
R_n^{\sels}\sim R_n^{\sols}\sim \dfrac12 R_n,\\
U_n^{\els}\sim U_n^{\ols}\sim 
U_n^{\rels}\sim U_n^{\rols}\sim
U_n^{\cels}\sim U_n^{\cols}&\sim 
U_n^{\sels}\sim U_n^{\sols}\sim \dfrac12 U_n.
\end{align*}
\end{thrm}


In contrast, the Alon-Tarsi conjecture \cite{AT92} 
asserts that $L_n^{\els}\ne L_n^{\ols}$ for even $n$. 
Two distinct generalisations of this are by
Zappa~\cite{Zap97}, who suggests that
$U_n^{\els}\ne U_n^{\ols}$ for all $n$ and Stones and Wanless~\cite{SW12}
who suggest that
$R_n^{\els}\ne R_n^{\ols}$ for all $n$.
These conjectures are only known to be true in some very special
cases (see \cite{Kot12b,Sto12,SW12} for details).

There is a natural action of $\sym_n\times\sym_n\times\sym_n$ on Latin
squares called {\em isotopism}.  Its orbits are called {\em isotopism
  classes}.  In essence, the reason that the Alon-Tarsi conjecture is
restricted to even orders is that parity is an isotopism class
invariant for even orders but not for odd orders.  Since it is known
that asymptotically almost all Latin squares have trivial stabiliser
in the group of isotopisms \cite{MW05}, a corollary of
\tref{t:asymequal} is that for even $n\iwb$ there are asymptotically
equal numbers of isotopism classes of Latin squares of each of the
possible parities.

\section{Parities are equinumerous}

Whenever we use the word ``random'' it will be implicit that we are
referring to the discrete uniform distribution (technically, actually a
sequence of such distributions as $n\iwb$).  

A {\em row cycle of length $\ell$} is a minimal (in the sense of
containment) non-empty $2\times\ell$ submatrix of a Latin square such
that each row of the submatrix contains the same symbols. We say that
a row cycle is {\em even} or {\em odd} depending on whether its length
$\ell$ is even or odd, respectively.  The two rows within a row cycle
can be switched to give a slightly different Latin square.  By
switching an odd row cycle we change the column parity and the symbol
parity, while leaving the row parity unchanged \cite{Wan04}. This
simple observation will be the key to our result.

Our aim is to find an odd row cycle that does not meet
the first row or first column.  We want to show that a random reduced
Latin square can be expected to have such a cycle. However, we do this
by first showing that a random Latin square also has such a
cycle. This allows us to employ techniques that may move beyond the
set of reduced Latin squares.
The techniques in question were developed in \cite{CGW08} to study
row cycle lengths in a random Latin square.  It will suit us to adapt
the results of \cite{CGW08}, which dealt with the first two rows, to
the last two rows instead. Similarly, \cite{CGW08} allowed
conditioning on the contents of a set $F$ of columns.  In that paper
$F$ was the last $n-m$ columns, however it will suit us to use a
variable set of columns which includes the first column. The results
from \cite{CGW08} apply unchanged, given the symmetry between
different columns, and between different rows.

We will consider random Latin squares of order $n$ as $n\iwb$. Soon we
will want to consider probabilities that are conditional on our Latin
square including a set $F$ of $n-m$ columns that includes the first
column. A prerequisite for the methods of \cite{CGW08} is that $F$
must contain entire row cycles in its last two rows. We impose the
extra condition that $F$ contains a single row cycle in its last two
rows. With this assumption it turns out that $F$ is unlikely to be too
big:

\begin{lemma}\label{l:nolongcyc}
With probability $1-o(1)$ a random Latin square of order $n$
has no cycle longer than $n-\log{n}$ within the last two rows.
\end{lemma}

\begin{proof}
Let $p$ be the probability that a random permutation in $\sym_n$ has
a cycle of length at least $n-\log{n}$. Then 
\begin{equation}\label{e:randpermlongcyc}
p=\frac{1}{n!}\sum_{i=\lceil n-\log n\rceil}^n\binom{n}{
i}(i-1)!(n-i)!
=\sum_{i=\lceil n-\log n\rceil}^n\frac{1}{i}
=O\bigg(\frac{\log n}{n}\bigg).
\end{equation}
Let $\xi$ be the multiset of the lengths of the row cycles in the
last two rows of a random Latin square of order $n$.
If $\xi$ has an element of size at least $n-\log{n}$,
then $\xi$ has at most $(\log{n})/2+1$ elements. 
Hence by \eref{e:randpermlongcyc} and \cite[Cor.\,4.5]{CGW08}, 
the probability that $\xi$ has an element of size at least $n-\log{n}$
is at most $n^{1/3}2^{(\log{n})/2+1}p=o(n^{-0.3})$.
\end{proof}


As foreshadowed, we now wish to condition on a random Latin square $L$
containing a set $F$ of entries consisting of entire columns
(including the first), where in the last two rows the entries of
$F$ form a single row cycle.  
This framework is consistent with \cite{CGW08}.  Let $m$ be
the number of columns that are not in $F$. Let $\rho$ be the partition
of $m$ formed by the lengths of the row cycles in the last two rows
that are not in $F$.  We will consider $m$ and $\rho$ to be discrete
random variables in the resulting probability space. Our results will
be phrased in terms of $m$ and $\rho$ but are otherwise independent of
$F$.  Note that with high probability $m\iwb$ as $n\iwb$, by
\lref{l:nolongcyc}. From \cite[Thm~4.9]{CGW08}, we have:

\begin{lemma}\label{l:fewcycs}
There exists a constant $c$ with $0<c<1$ such that
$\rho$ has fewer than $9\sqrt{m}$ parts 
with probability $1-o(c^m)$.
\end{lemma}



Let $P(m)$ denote the partitions of $m$ into parts of size at least
$2$. Let 
\begin{equation}\label{e:gamma}
\gamma(\lambda)=\frac{m!}{\displaystyle\prod_{i=2}^m
\lambda_i!\,i^{\lambda_i}}
\end{equation}
be the number of derangements with cycle structure 
$\lambda=(2^{\lambda_2},3^{\lambda_3},\dots,m^{\lambda_m})\in P(m)$.
Here and henceforth, $i^{\lambda_i}$ denotes $\lambda_i$ parts of size $i$,
where we allow the possibility that $\lambda_i=0$.  
Let $S(\lambda,F)$ denote the set of
Latin squares that contain $F$ and have $\rho=\lambda$.

\begin{lemma}\label{l:mulam}
Let $m$ be even and suppose
that $\lambda=(2^{\lambda_2},4^{\lambda_4},\dots,m^{\lambda_m})\in P(m)$
has only even parts, including one of size $z$ where $z\iwb$
as $n\iwb$.  Let
$M$ be the set of $\mu\in P(m)$ such that 
$\mu$ is obtained from $\lambda$ by splitting 
a part of size $z$ into two parts of odd size.
Then $$\sum_{\mu\in M}|S(\mu,F)|=|S(\lambda,F)|\,\Omega(\log{z}).$$
\end{lemma}

\begin{proof}
Let $\mu=(2^{\mu_2},3^{\mu_3},\dots,m^{\mu_m})\in M$ be
such that $\mu$ is obtained from 
$\lambda$ by splitting one part of size $z$ into parts of size $a$ and
$z-a$, where $a$ is odd (and thus $z-a$ is too) and $a<z-a$. Since $\lambda$ has
only even parts, $\lambda_a=\lambda_{z-a}=0$ and
$\mu_a=\mu_{z-a}=1$. Moreover $\lambda_{z}=\mu_z+1\geq 1$.
By \eref{e:gamma},
$$\frac{\gamma(\mu)}{\gamma(\lambda)}=\frac{z\lambda_z}{a(z-a)}.$$
By \cite[Lem~3.13]{CGW08}, this implies that
$${|S(\mu,F)|}\geq \frac{2z\lambda_z|S(\lambda,F)|}{3a(z-a)}
\geq \frac{2z|S(\lambda,F)|}{3a(z-a)}.$$
Thus,
$$\sum_{\mu\in M}|S(\mu,F)|
\geq \frac{2}{3}|S(\lambda,F)|z\sum_{a=1}^{w}
\frac{1}{(2a+1)(z-2a-1)},$$
where $w=\lfloor(z-3)/4\rfloor$ is the largest integer satisfying $2w+1<z-2w-1$.
However, $1/((2x+1)(z-2x-1))$, is a decreasing function of $x$ for
$1\le x\le w$ so
\begin{align*}
\sum_{a=1}^{w}
\frac{1}{(2a+1)(z-(2a+1))}
& \geq \int_{1}^{w} \frac{dx}{(2x+1)(z-2x-1)} \\
& = \frac{1}{2z}\log\frac{(2w+1)(z-3)}{3(z-2w-1)}
= \Omega\left(\frac{\log{z}}{z}\right), 
\end{align*}
from which the result follows.
\end{proof}

We next show that with high probability there is an odd cycle that
does not meet the first row or column (assuming $n>2$). We deduce
this first for general Latin squares, then infer it for 
reduced Latin squares.

\begin{thrm}
With probability $1-o(1)$ there is a part of odd size in $\rho$.
\end{thrm}

\begin{proof}
By \lref{l:nolongcyc} we know that $m\ge\log n$ with
probability $1-o(1)$. 
By \lref{l:fewcycs}, there asymptotically
almost surely are at most $9\sqrt{m}$ 
parts in $\rho$, so there is some part of size at least
$\sqrt{m}/9$. 
By \lref{l:mulam}, the probability of $\rho$ having no odd parts
is at most $O(1/\log m)=o(1)$, as claimed.
\end{proof}


\begin{corol}\label{cy:oddred}
With probability $1-o(1)$, in the last two rows of a random
reduced Latin square of order $n$ there is a cycle of odd length that
does not include the first column.
\end{corol}

\begin{proof}
We can reduce a Latin square by permuting the symbols so that the first
column is in order, then permuting the columns so that the first row is in
order. These operations do not affect whether the last two rows 
contain a cycle of odd length that does not include the first column
(note that the first column does not move). Also, each reduced Latin square
is produced the same number of times, namely $n!(n-1)!$ times, when the
above reduction is applied to all Latin squares. So reduced Latin squares
have the same probability of having the property of interest as general
Latin squares do.
\end{proof}

We are now in a position to prove our main result. As already noted,
by switching an odd row cycle we change the column parity and the
symbol parity.  Hence \cyref{cy:oddred} provides us with an
involution, which acts on all but a negligible fraction of reduced
Latin squares, and which reverses column parity and reverses symbol
parity. It follows that $R^{\cels}_n\sim R^{\cols}_n$ and
$R^{\sels}_n\sim R^{\sols}_n$. \Tref{T:relat} then tells us that,  
\begin{align*}
&R_n^{000}\sim R_n^{011}\sim R_n^{101}\sim R_n^{110} \sim \dfrac14 R_n\quad\hbox{\rm for }n\equiv0,1\pmod 4,\\
&R_n^{111}\sim R_n^{100}\sim R_n^{010}\sim R_n^{001} \sim \dfrac14 R_n\quad \hbox{\rm for }n\equiv2,3\pmod 4,
\end{align*}
as $n\iwb$. The remainder of \tref{t:asymequal} 
can then be easily deduced from \Tref{T:relat} and the following
additional observations. Replacing each row of a Latin square by its
inverse (when considered as a permutation) converts 
reduced Latin squares into normalised unipotent Latin squares and 
vice versa. Hence $R^{abc}_n=U^{acb}_n$ for all parities $\pi=abc$ and all
$n$. We also know two more facts from \cite{SW12}. Firstly, for even $n$,
\begin{equation*}
L^{P}_n = n!\, (n-1)!\, R^{P}_n = n!\, (n-1)!\, U^{P}_n,
\end{equation*}
whenever
$P \in \{\els,\ols,\rels,\rols,\cels,\cols,\sels,\sols\}$ or $P$ is any parity.
Secondly, for odd $n\ge3$,
\begin{equation*}
L_n^{000}=L_n^{011}=L_n^{101}=L_n^{110}\hbox{ and }
L_n^{111}=L_n^{100}=L_n^{010}=L_n^{001}.
\end{equation*}

\section{Concluding comments}

We have confirmed a conjecture from \cite{SW12} and explained why the
large components have comparable size in the switching graphs studied
in \cite{Wan04}. Our results do not explain why the components in 
the switching graphs in \cite{KMOW14} have comparable size. At this stage
we have no tools to study the lengths of cycles in random {\em symmetric}
Latin squares.

A much stronger result than \tref{t:asymequal} seems very likely
to be true.  By Wilf~\cite[p.209]{Wil94}, the proportion of
permutations in $\sym_n$ that have no odd cycles is
$2^{-n}n!/(n/2)!^2\sim \sqrt{2/(\pi n)}$. It follows immediately that
the proportion of derangements with no odd cycles is also
$O(n^{-1/2})$. 
Hence the proportion of derangements with at most one odd cycle is
$O(n^{-1/2})$ if $n$ is even and no more than
\begin{equation*}
\frac{(n{-}1)!}{n!}
+
\frac{O(1)}{n!}
\sum_{i=0}^{(n-3)/2}{n\choose 2i+1}
\frac{(2i)!(n-2i-1)!}{(n-2i-1)^{1/2}}
= \frac{1}{n}
+
O(1)\!\!\sum_{i=0}^{(n-3)/2}\frac{1}{(2i+1)(n-2i-1)^{1/2}}
\end{equation*}
if $n$ is odd. Approximating the sum by an integral, we find that 
for all $n$ the 
proportion of derangements with at most one odd cycle is $O(n^{-1/2}\log n)$. 
If \cite[Conj.\,6.1]{CGW08} holds then a similar statement would be true
about the cycles in the last two rows of a Latin square: namely
there would be at least two odd cycles with probability
$1-O(n^{-1/2}\log n)$.  At least one of these cycles is switchable in
the sense that it does not hit the first column.  Amongst the squares
with no switchable odd cycle in the last two rows, we can look for a
switchable odd cycle in rows $n-3,n-2$, then in rows $n-5,n-4$ and so
on up to, but not including, the first row. Switching the first
switchable cycle that we find in this way would give us an involution,
because switching cycles in rows $x$ and $y$ never affects the cycle
lengths between rows other than $x$ and $y$.  The domain of the
involution includes all reduced Latin squares that have any switchable
cycle in an appropriate pair of rows. It seems plausible that each pair
of rows would have a switchable cycle with something close to an
independent probability $1-O(n^{-1/2}\log n)$, meaning the
proportion of reduced Latin squares
outside the domain of our involution would be $O(n^{-cn})$
for some constant $c>0$. 
Hence for each given $n$ the numbers of reduced Latin squares with
each of the four possible parities are probably very much closer 
to each other than our work has demonstrated.

  \let\oldthebibliography=\thebibliography
  \let\endoldthebibliography=\endthebibliography
  \renewenvironment{thebibliography}[1]{%
    \begin{oldthebibliography}{#1}%
      \setlength{\parskip}{0.4ex plus 0.1ex minus 0.1ex}%
      \setlength{\itemsep}{0.4ex plus 0.1ex minus 0.1ex}%
  }%
  {%
    \end{oldthebibliography}%
  }

\end{document}